\documentclass[12pt]{amsart}

\usepackage{amssymb}
\usepackage{graphicx}
\usepackage{paralist}
\usepackage{url}
\usepackage{psfrag}
\usepackage{auto-pst-pdf}
\usepackage{comment}

\newtheorem{theorem}{Theorem}[section]
\newtheorem{maintheorem}[theorem]{Main Theorem}
\newtheorem{proposition}[theorem]{Proposition}
\newtheorem{lemma}[theorem]{Lemma}
\newtheorem{corollary}[theorem]{Corollary}

\theoremstyle{definition}

\newtheorem{conjecture}[theorem]{Conjecture}

\newenvironment{proof31}{\paragraph{\emph{Proof of Proposition 3.1}}}{\hfill$\square$}

\newcommand{\br}{\rm{br}}
\newcommand{\lk}{\rm{lk}}
\newcommand{\stellar}{\rm{stellar}}

\begin{document}
\title{Stellar theory for flag complexes}

\noindent
\date{}

\author{Frank H.~Lutz}
\address{Institut f\"ur Mathematik,
Technische Universit\"at Berlin,
Stra{\ss}e des 17.~Juni~136,
10623 Berlin, Germany
}
\email{lutz@math.tu-berlin.de}

\thanks{Research of the first author was supported by the DFG Research Group ``Polyhedral Surfaces'', 
by \textsc{VILLUM FONDEN} through the Experimental Mathematics Network and by the Danish National Research Foundation (DNRF) through the Centre for Symmetry and Deformation.}


\author{Eran Nevo}
\address{
Department of Mathematics,
Ben Gurion University of the Negev,
Be'er Sheva 84105, Israel
}
\email{nevoe@math.bgu.ac.il}

\thanks{Research of the second author was partially supported by Marie Curie grant IRG-270923 and ISF grant.}

\maketitle
\begin{abstract}
Refining a basic result of Alexander, we show that two \emph{flag} simplicial complexes are piecewise linearly homeomorphic if and only if they can be connected by a sequence of flag complexes, each obtained from the previous one by either an edge subdivision or its inverse.
For flag spheres we pose new conjectures on their combinatorial structure forced by their face numbers, analogous to the extremal examples in the upper and lower bound theorems for simplicial spheres. Furthermore, we show that our algorithm to test the conjectures searches through the entire space of flag PL spheres of any given dimension.
\end{abstract}

\section{Introduction}\label{sec:intro}
A basic result in piecewise linear (PL) topology, is that
\begin{theorem}\label{thm:stellar}(Alexander, \cite[Theorem 15:1]{Alexander1930})
Two simplicial complexes are PL homeomorphic if and only if they can be connected by a sequence of stellar subdivisions and their inverses.
\end{theorem}
See e.g. \cite[Theorem 4.5]{Lickorish:PLsurvey} for a modern proof and further references.

An (abstract) simplicial complex is called \emph{flag} if all its minimal non-faces (called also \emph{missing faces}) have cardinality two; equivalently, it is the complex of cliques of a simple graph.
Flag complexes arise in many mathematical contexts, and often interesting families of flag complexes share the same PL type; for example, the order complexes of intervals with respect to Bruhat order on Coxeter groups are PL spheres \cite{Bjorner84:PosetsRegularCWBruhatOrder}.
Very recently Adiprasito and Benedetti showed that the Hirsch conjecture, on the diameter of the facet-ridge graph, holds for all (connected) flag homology manifolds \cite{Adiprasito-Bennedetti:Hirsch}.

Our main result says that:
\begin{maintheorem} \label{thm:main}
Two flag simplicial complexes are PL homeomorphic if and only if they can be connected by a sequence of edge subdivisions and their inverses such that all the complexes in the sequence are flag.
\end{maintheorem}

\pagebreak

Equivalently, in graph language, this theorem reads as:
\begin{corollary}
The clique complexes of two graphs $G$ and $G'$ are PL homeomorphic if and only if there is a sequence of graphs $G=G_0,G_1,\dots,G_t=G'$ such that for any $1\leq i\leq t$, one of $G_{i-1},G_i$ is obtained from the other by placing a new vertex $v$ at the middle of an edge $\{a,b\}$ (breaking it into two edges) and connecting $v$ to all common neighbors of $a$ and $b$.
\end{corollary}

Along the way, in Proposition~\ref{prop:br}, we will show that one can connect any simplicial complex to its barycentric subdivision by a sequence of edge subdivisions (no inverse moves are needed). We use this result to rediscover Alexander's result \cite[Corollary 10:2d]{Alexander1930} that in Theorem \ref{thm:stellar} subdivisions at edges suffice; see Corollary~\ref{cor:PLedge-sd}.

We explain now an aspect in which our proof is advantageous. In view of Corollary~\ref{cor:PLedge-sd}, one may strengthen Alexander's conjecture that in Theorem~\ref{thm:stellar} one can perform all stellar subdivisions before all the inverse stellar subdivisions (see e.g. \cite[p.~14, unsolved problem]{Hudson}) as follows:
\begin{conjecture}\label{conj:edgeAlexanderConj}
Two simplicial complexes $\Delta$ and $\Delta'$ are PL homeomorphic
if and only if  they have a common refinement by a sequence of edge subdivisions from each of them.
\end{conjecture}

Our proof of Corollary~\ref{cor:PLedge-sd} shows that Conjecture~\ref{conj:edgeAlexanderConj} is true if $\Delta'$ is obtained from~$\Delta$ by some stellar subdivision (while Alexander's proof connects them by a ``zigzag" sequence).
For further development on Conjecture~\ref{conj:edgeAlexanderConj} and its connection to the strong Oda conjecture see \cite{DaSilva-Karu} and the references therein.

We summarize Alexander's results and our main theorem in the language of graph theory.
Let $\Delta$ be a simplicial complex, and define an (infinite) graph $G_s(\Delta)=(V,E)$ as follows.
Let~$V$ be the set of simplicial complexes PL homeomorphic to $\Delta$, and $\{\Delta',\Delta''\}\in E$ if and only if one of the complexes $\Delta'$ and $\Delta''$ is obtained from the other by a stellar subdivision,
say at a face $F$.
Let $G_e(\Delta)$ be the graph obtained from $G_s(\Delta)$ by  deleting the edges for which $1<{\rm dim}\, F:=|F|-1$.
Let $G_f(\Delta)$ be the graph induced from $G_e(\Delta)$ by restricting to the vertices corresponding to flag complexes.
Then $G_f(\Delta)\subseteq G_e(\Delta)\subseteq G_s(\Delta)$ satisfy:
\begin{itemize}
\item{}$G_s(\Delta)$ is connected (Alexander \cite{Alexander1930}).
\item{}$G_e(\Delta)$ is connected (Alexander \cite{Alexander1930}).
\item{}$G_f(\Delta)$ is connected (Theorem \ref{thm:main}).
\end{itemize}

Next, we consider flag spheres, and pose two new conjectures about the combinatorial structure forced by their face numbers, analogous to the extremal examples in the upper and lower bound theorems for simplicial spheres. The conjectures are supported by computer experiments --- as a consequence of the Main Theorem~\ref{thm:main} our algorithm searches through the entire space of flag PL spheres of any fixed dimension; see Corollary \ref{cor:cumputer_program}.

Section~\ref{sec:prelim} provides preliminaries on stellar theory. Barycentric subdivisions are discussed in Section~\ref{sec:br},
concluding that $G_e(\Delta)$ is connected in Section~\ref{sec:stellar}.
Section~\ref{sec:flag} gives the proof that $G_f(\Delta)$ is connected, and conjectures for extremal flag spheres are formulated in Section~\ref{sec:apps}.

\section{Preliminaries}\label{sec:prelim}
A (finite) \emph{abstract simplicial complex} on a (finite) set of vertices $V$
is a system $\Delta\subseteq 2^V$ of subsets of $V$ such that
for every $F\in\Delta$ and $F'\subseteq F$ also $F'\in\Delta$.
An element $F\in\Delta$ is called a \emph{face} of $\Delta$, an inclusion maximal face
is a \emph{facet},
and we use set operations $F\cup F'$, $F\uplus F'$, $F\cap F'$, $F\backslash F'$, and $|F|$
to denote unions, disjoint unions, intersections, differences, and cardinalities of faces, respectively.

For a simplicial complex $\Delta$ and a face $F$ in it, let the \emph{stellar subdivision} of $\Delta$ at~$F$ be
$$\stellar_{\Delta}(F):=\{F' \in \Delta: F\cap F'\neq \emptyset\} \cup (\{v_F\} \ast \partial F \ast \lk_{\Delta}(F)).$$
Here, $\lk$ denotes the \emph{link} of a face,
$$\lk_{\Delta}(F)=\{F'\in \Delta: F\cap F'=\emptyset, F'\cup F\in \Delta\},$$
$*$ the \emph{join product} of two simplicial complexes with disjoint vertex sets,
$$\Delta*\Delta'=\{F\cup F':F\in\Delta,F'\in\Delta'\},$$
and $\partial$ the \emph{boundary complex} of a face,
$$\partial F =\{F': F'\subseteq F, F'\neq F\},$$
and $v_F$ is a vertex not in~$\Delta$.

Consider a \emph{geometric realization} $||\Delta||$ of $\Delta$, that is, a geometric simplicial complex $||\Delta||$ isomorphic to $\Delta$
in some ${\mathbb R}^n$.
Geometrically, placing the new vertex $v_F$ anywhere in the relative interior of $||F||$ and
taking convex hulls of $v_F$ with the faces of $\partial F$ and the simplices in $\lk_{\Delta}(F)$ yields the same embedded space for the geometric realization $||\stellar_{\Delta}(F)||$ as $||\Delta||$.

Let $\br(\Delta)$ denote the \emph{barycentric subdivision} of $\Delta$, namely the simplicial complex whose vertices are
indexed by the nonempty faces of $\Delta$ and whose simplices correspond to a set of faces forming a chain with respect to inclusion.
To get the same embedded space for the geometric realizations of $||\Delta||$ and
$||\br(\Delta)||$, for each nonempty face $F\in \Delta$ place $v_F$ at the barycenter of $||F||$ in the embedding induced by $||\Delta||$.
It is known that totally ordering the faces of  $\Delta$ by decreasing dimension and performing stellar subdivions according to this order changes $\Delta$ to $\br(\Delta)$.

\section{Barycentric subdivision: edges suffice}\label{sec:br}


\begin{figure}
    \begin{center}
      \begin{postscript}
      \psfrag{v0}{$v_0$}
      \psfrag{v1}{$v_1$}
      \psfrag{v2}{$v_2$}
      \psfrag{v}{$v_3$}
      \psfrag{u0}{$u_0$}
      \psfrag{u1}{$u_{01}$}
      \psfrag{u2}{$u_{012}$}
      \psfrag{u}{$u_{013}$}
      \psfrag{u'}{$u_{02}$}
      \psfrag{u''}{$u_{03}$}
      \psfrag{u'''}{$u_{12}$}
      \psfrag{u''''}{$u_{13}$}
      \psfrag{v0v1}{$v_0v_1$}
      \psfrag{v0v2}{$v_0v_2$}
      \psfrag{v0v}{$v_0v_3$}
      \psfrag{v1v2}{$v_1v_2$}
      \psfrag{v1v}{$v_1v_3$}
      \psfrag{v0v1v2}{$v_0v_1v_2$}
      \psfrag{v0v1v}{$v_0v_1v_3$}
      \psfrag{e}{$\emptyset$}
      \includegraphics[width=.425\linewidth]{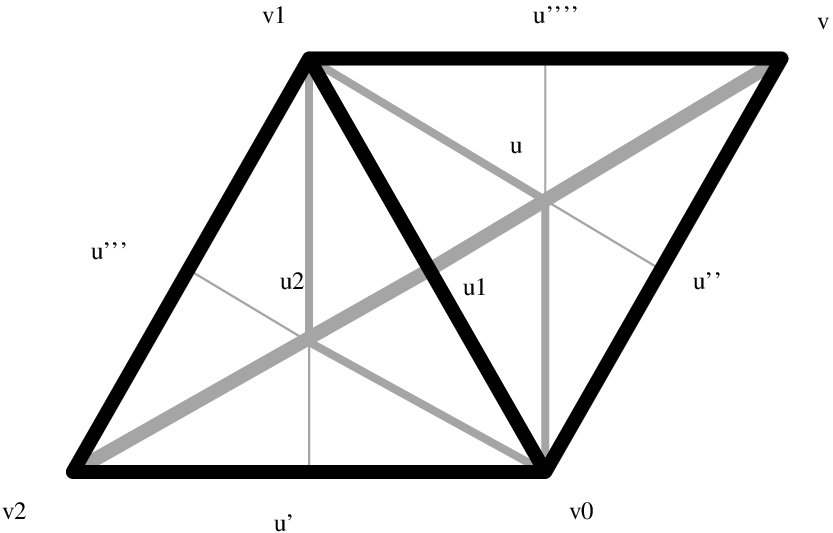}\hspace{.05\linewidth}\includegraphics[width=.425\linewidth]{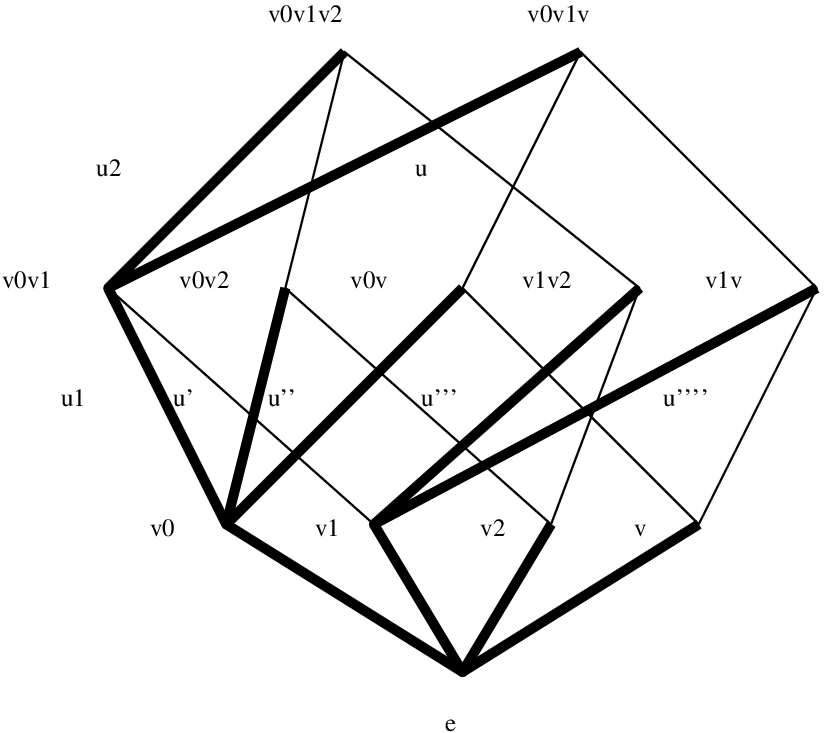}
      \end{postscript}
    \end{center}
    \caption{Iterated edge subdivisions for two triangles  $v_0v_1v_2$ and $v_0v_1v_3$ according to a spanning tree in the Hasse diagram
                   of the two triangles.
                   The new vertices are inserted in the lexicographic order $u_{01},u_{012},u_{013},u_{02},u_{03},u_{12},u_{13}$.}
      \label{fig:backtrack-sd}
  \end{figure}


\begin{proposition}\label{prop:br}
Let $\Delta$ be a simplicial complex, and $\br(\Delta)$ denote its barycentric subdivision. Then there is a sequence of edge subdivisions from $\Delta$ to $\br(\Delta)$.
\end{proposition}
First, we describe an algorithm for producing such a sequence, then in Lem\-ma~\ref{claim:backtrack} we prove its correctness.
Choose a maximal chain of simplices in $\Delta$, $\emptyset=F_{-1}\subseteq F_0\subseteq F_1\subseteq\cdots \subseteq F_t$, with $\dim(F_i)=i$ and $\dim(\Delta)=t$.
Denote $\{v_i\}=F_i \setminus F_{i-1}$ for $0\leq i\leq t$ and subdivide the edge $F_1$ by a new vertex $u_{01}$. Continue to subdivide the edges $\{u_{01\dots i-1},v_i\}$ by a new vertex $u_{01\dots i}$ for $1<i\leq t$.
Now backtrack by replacing $F_t$ by another $t$-simplex $F_t'= F_{t-1}\cup \{v_{t+1}\}$, if it exists, and subdivide $\{u_{01\dots t-1},v_{t+1}\}$ by $u_{01\dots t-1\,t+1}$. Keep the backtracking and edge subdivision process until a (unique) new vertex is added for each simplex in $\Delta$ of positive dimension.

This process is conveniently described as choosing a spanning tree in the Hasse diagram of the face poset of $\Delta$ by a backtracking depth first search --- the depth of a node equals its rank in the poset, and for pairs $(\emptyset\subseteq \rm{vertex})$ the edge subdivision part is empty. (For example, by this rule all edges in $\Delta$ containing the vertex $v_0$  are subdivided before the other edges in $\Delta$; this property is not important, as the next lemma will show, it just eases the description of the backtracking process.)

Figure \ref{fig:backtrack-sd} gives an illustration for the subdivision procedure in the case that $\Delta$ has
exactly two triangular facets $\{v_0,v_1,v_2\}$ and $\{v_0,v_1,v_3\}$, where, for short we write $v_0v_1v_2$ and $v_0v_1v_3$,
respectively. We first process the triangle $v_0v_1v_2$, where we proceed dimensionwise, as indicated by the indices.
We then backtrack to the edge $v_0v_1$, which is included in the second triangle $v_0v_1v_3$, and subdivide the edge $u_{01}v_3$
by inserting a vertex $u_{013}$.
Once both triangles are processed, we have to backtrack to $v_0$ and then subdivide the edges $v_0v_2$ and $v_0v_3$
by placing vertices $u_{02}$ and $u_{03}$, respectively. Next, we backtrack to $\emptyset$ and then go up again to $v_1$
to finally subdivide the edges $v_1v_2$ and $v_1v_3$ by inserting vertices $u_{12}$ and $u_{13}$, respectively.

We claim that the resulting complex equals $\br(\Delta)$, regardless of the choices made during the backtracking process.
This as a special case of the following lemma. 

\begin{lemma}\label{claim:backtrack}
Let $s$ be a sequence of stellar subdivisions starting from a simplicial complex $\Delta$, ending at $s(\Delta)$, and satisfying:
\begin{compactitem}
\item[(i)] For any face $F\in \Delta$ with $\dim(F)>0$ there is a unique vertex $v_F\in s(\Delta)$, located at the barycenter of $||F||$ (note that possibly $v_F$ is added for a stellar subdivision \emph{not} at $F$, but at a face $||G||\subseteq||F||,\ G\notin \Delta$ that has been introduced by some earlier subdivision);
and
\item[(ii)] if $F',F''\subseteq F$ are three faces in $\Delta$ of positive dimension, and if $F'$ and $F''$ are incomparable,
then $v_F$ does not appear later then both $v_{F'}$ and $v_{F''}$ in $s$.
\end{compactitem}
Then $s(\Delta)$ is combinatorially isomorphic to $\br(\Delta)$.
\end{lemma}
\begin{proof}
First, we reduce the problem to the case where $\Delta=\overline{V}:=\{F: F\subseteq V\}$ is a simplex.
For this, let $W\subseteq V$ be a subset of the vertices of a general complex $\Delta$. Then the effect of a stellar subdivision of $\Delta$ at a face $F$ on the induced complex $\Delta[W]$ is nothing if $F$ is not a subset of $W$ and equals $\stellar_{\Delta[W]}(F)$ if $F\subseteq W$. Moreover, the restriction $s_W$ of the sequence $s$ to $\Delta[W]$ satisfies conditions (i) and (ii) in the lemma. Thus, by choosing $W$ to be the vertex set of a face in $\Delta$, we see that the lemma will follow if it is true for any simplex $\overline{V}$.

Assume $\Delta=\overline{V}$ and we prove the lemma by induction on $\dim(\overline{V})$, where the case $\dim(\overline{V})\leq 1$ is trivial.
Thus, assume $\dim(\overline{V})>1$.
By the induction hypothesis and the remark above on $s_W$ (for all strict subsets $W$ of $V$), we get that the sequence $s$ changes $\partial \overline{V}$ to $\br(\partial \overline{V})$ (note that $v_V$ has no effect on the subdivision of $\partial \overline{V}$).

As the geometric realizations of $\br(\overline{V})$ and $s(\overline{V})$ give the same space, it is enough to show that any facet of $s(\overline{V})$ is also a facet of $\br(\overline{V})$. As the restriction of $s(\overline{V})$ to $||\partial \overline{V}||$ is $\br(\partial \overline{V})$, it is enough to show that
\begin{compactitem}
\item[(*)] for any initial subsequence $s'$ of $s$ that contains $v_V$, all facets of $s'(\overline{V})$ are of the form $\{v_V\}\cup F$ where $F$ is a facet of $s'(\partial \overline{V})$.
\end{compactitem}
To prove (*), notice that all vertices $v_F$ that appear before $v_V$ in $s$ correspond to pairwise comparable faces by (ii), hence these faces form a chain of faces in $\overline{V}$, say with a maximal face $F'$.

Denote by $s^{F'}$ the initial part of $s$ up to vertex $v_{F'}$, and by $s^{F'}(L)$ the restriction of $s^{F'}(\Delta)$ to $||L||$, where $L$ is a subcomplex of $\Delta$.

By induction on dimension, (*) holds for $F'$, thus all facets in $s^{F'}(\overline{F'})$ are of the form $\{v_{F'}\}\cup F''$ where $F''$ is a facet of $s^{F'}(\partial\overline{F'})$.
Also, all vertices $v_F$ appearing before $v_{F'}$ satisfy $F\subseteq F'$.
Hence, all the facets in $s^{F'}(\overline{V})$ are of the form
$\{v_{F'}\}\cup F'' \cup (V\backslash F')$ where $F''$ is a facet of $s^{F'}(\partial F')$, thus they contain the face $F'''=\{v_{F'}\}\cup (V\backslash F')$.
Note that $||\overline{F'''}||$ contains the barycenter of $||\overline{V}||$ and $F'''$ is the minimal face of $s^{F'}(\overline{V})$ with this property.
Thus, $v_V$ in $s$ corresponds to a stellar subdivision of $s^{F'}(\overline{V})$ at $F'''$, and the resulting complex $s^{V}(\overline{V})$ has the property that all its facets have the form $\{v_V\}\cup F''''$ where $F''''$ is a facet of $s^{V}(\partial \overline{V})$.
By (i), any vertex in $s$ that appears after $v_V$ corresponds to a stellar subdivision at a face $F$ contained in $||\partial \overline{V} ||$
and hence all the facets that contain $F$ also contain $v_V$,
thus all facets after the subdivision contain $v_V$ and (*) follows.
\end{proof}
\begin{proof31}
Our algorithm described above respects the conditions of Lemma~\ref{claim:backtrack}, from which correctness follows. 
\end{proof31}

\section{Stellar theory: edges suffice}\label{sec:stellar}

\begin{corollary}(Alexander, \cite[Corollary 10:2d]{Alexander1930})\label{cor:PLedge-sd}
If $\Delta$ and $\Gamma$ are PL homeomorphic simplicial complexes, then they are connected by a sequence of edge subdivisions and their inverses.
\end{corollary}
We give a proof based on Proposition \ref{prop:br}, whose advantage we explained in the introduction.

\begin{proof}
By Theorem \ref{thm:stellar} it suffices to prove the case where $\Gamma$ is obtained from $\Delta$ by a stellar subdivision at a face $F$.

Let $s(\overline{F})$ be a sequence of edge subdivisions in the simplex $\overline{F}$, from $\overline{F}$ to $\br(\overline{F})$ as guaranteed by Proposition \ref{prop:br}. Performing $s(\overline{F})$ starting from $\Delta$ ends in a simplicial complex, denote it $\Delta'$.
Let $s(\partial \overline{F})$ be a sequence of edge subdivisions in the boundary complex $\partial \overline{F}$, from $\partial \overline{F}$ to $\br(\partial \overline{F})$ as guaranteed by Proposition \ref{prop:br}. Performing $s(\partial \overline{F})$ starting from $\Gamma$ ends in a simplicial complex, denote it $\Gamma'$.

To finish the proof we show that $\Delta'\cong \Gamma'$ (or equality, with the obvious identifications of vertices given by geometric location at barycenters --- which we will use below).
Considering the effect of a stellar subdivision on geometric realizations, with each (closed) face $F'$ of the original complex there is a canonically associated closed ball consisting of a subcomplex in the resulting complex, whose underlying space is $||\overline{F'}||$.
The face $F'\in \Delta$ has a unique decomposition $F'=F'_+\cup F'_-$ such that $F'_+\subseteq F$ and $F'_-\cap F=\emptyset$.

Then, as stellar subdivision and join commute (namely for disjoint simplicial complexes $\Delta_I,\Delta_{II}$ and a face
$F_I\in \Delta_I$, $\stellar_{\Delta_I*\Delta_{II}}(F_I)=\stellar_{\Delta_I}(F_I)*\Delta_{II}$),
we get that for $F'\in \Delta$ the complex associated with $||\overline{F'}||$ in $\Delta'$ is $\overline{F'_-}* \br(\overline{F'_+})$. If $F'_+\neq F$, then $F'\in \Gamma$ and again $\overline{F'_-}* \br(\overline{F'_+})$ is the corresponding subcomplex in $\Gamma'$.
If $F'_+=F$, denote by $v_F$ the vertex in the relative interior of $||\overline{F}||$ in the geometric realizations of
both (by the abuse of notation explained above)
$\Delta'$ and $\Gamma$ (and $\Gamma'$).
Then the subcomplex corresponding to $||\overline{F'}||$ is as follows: in $\Delta'$ it is $\overline{F'_-}* \br(\overline{F})=\overline{F'_-}*\{v_F\}*\br(\partial \overline{F})$; in $\Gamma$ it is $\overline{F'_-}*\{v_F\}*\partial \overline{F}$, hence in $\Gamma'$ it is $\overline{F'_-}*\{v_F\}*\br(\partial \overline{F})$.
\end{proof}

\section{Flag complexes: edges suffice}\label{sec:flag}
Recall that  a \emph{missing face} of a simplicial complex $\Delta$ is a subset $F$ of vertices of a $\Delta$
satisfying $F\notin\Delta$ and $\partial \overline{F}\subseteq \Delta$, and that $\Delta$ 
is \emph{flag} if all its missing faces have cardinality two.

We now describe an invariant to measure how `close' some simplicial complex is to a flag complex.
Define
$$d(\Delta):=\sum_{\ F\notin\Delta,\partial \overline{F}\subseteq \Delta, |F|>2}  |F|, $$
thus $\Delta$ is flag if and only if $d(\Delta)=0$.
The following observation will be important.

\begin{lemma}\label{lem:d(Delta)}
Let $\Delta'$ be obtained from a simplicial complex $\Delta$ by an edge subdivision, and that edge is contained in a missing face of $\Delta$ of dimension at least $2$. Then $d(\Delta')<d(\Delta)$.
\end{lemma}
\begin{proof}
Let $\{a,b\}$ be the edge subdivided, by a new vertex $v$.
The missing faces of $\Delta'$ are obtained from the missing faces of $\Delta$ as follows: 
if (the disjoint union) $F\uplus\{a,b\}$ is missing in $\Delta$ replace it by $F\cup\{v\}$
(of smaller size), the other missing faces of $\Delta$ are missing also in $\Delta'$, and the rest of the missing faces of $\Delta'$ are of the form $\{v,u\}$ for some vertex~$u$.

As missing edges do not effect $d(\cdot)$, and $\Delta$ has a missing face of the form $F\uplus\{a,b\}$ with $F$ nonempty, we have $d(\Delta')<d(\Delta)$.
\end{proof}
The argument above on missing faces also verifies that
\begin{lemma}\label{lem:edge-sd(flag)}
Let $\Delta'$ be obtained from a simplicial complex $\Delta$ by an edge subdivision.
If the edge subdivided is in no missing face, then $d(\Delta')=d(\Delta)$. In particular, if $\Delta$ is flag, then
$\Delta'$ is flag.
$\square$
\end{lemma}

\setcounter{section}{1}
\setcounter{theorem}{1}

\begin{maintheorem}
Two flag simplicial complexes $\Delta$ and $\Gamma$ are PL homeomorphic if and only if they can be connected by a sequence of edge subdivisions and their inverses such that all the complexes in the sequence are flag.
\end{maintheorem}\setcounter{section}{5}\setcounter{theorem}{0}
\begin{proof} The `if' part is obvious. As for the `only if' part, by Corollary~\ref{cor:PLedge-sd}, there is a sequence $\alpha$ of simplicial complexes $\Delta=\Delta_0, \Delta_1,\dots,\Delta_t=\Gamma$ such that for each $1\leq i\leq t$, one of $\Delta_i$ and $\Delta_{i-1}$ is obtained from the other by an edge subdivision.
However, not all complexes in $\alpha$ are flag. We now show how to modify $\alpha$ to a new sequence from $\Delta$ to $\Gamma$ where each $\Delta_i$ is flag. The modification is done in steps, where at each step the invariant $d(\cdot)$ is improved, until a sequence of flag complexes is obtained.

For a sequence $\alpha$ as above let $\max(\alpha):=\max_{0\leq i\leq t}d(\Delta_i)$.
In the case $\max(\alpha)>0$ let  $\rm{mult}(\alpha)$ be the number of $i$'s for which $d(\Delta_i)=\max(\alpha)$, and define ${\bf d}(\alpha):=(\max(\alpha),\rm{mult}(\alpha))$.
Equip $\mathbb{N}^2$ ($\mathbb{N}=\{1,2,3,\dots\}$) with the lexicographic order, namely $(a,b)<(c,d)$ if and only if either $a<c$ or $a=c$ and $b<d$,
and append to it a new element $\hat{0}$, smaller then all, to get a linear order $P$ with a minimum~$\hat{0}$.
Define ${\bf d}(\alpha)=\hat{0}$ if $\max(\alpha)=0$.
Thus, $\alpha$ is a sequence as required if and only if ${\bf d}(\alpha)=\hat{0}$.

Next, we modify the sequence $\alpha$. Assume ${\bf d}(\alpha)>\hat{0}$, as else we are done.
Call index $i$ a \emph{valley} of $\alpha$ ($0<i<t$) if each of $\Delta_{i-1}$ and $\Delta_{i+1}$ is obtained from $\Delta_{i}$ by an edge subdivision.
As both $\Delta_{0}$ and $\Delta_{t}$ are flag, by the assumption ${\bf d}(\alpha)>\hat{0}$ and Lemma~\ref{lem:edge-sd(flag)}, $\alpha$ has a valley. By Lemma \ref{lem:d(Delta)}, $\alpha$ has a valley $i$
such that $d(\Delta_i)=\max(\alpha)$. Consider such $i$, and let $e_1$ (resp. $e_2$) be the edge of $\Delta_i$ subdivided to obtain $\Delta_{i-1}$ (resp.\ $\Delta_{i+1}$). Without loss of generality, we assume $e_1\neq e_2$, since otherwise $\Delta_i$ along with either $\Delta_{i-1}$ or $\Delta_{i+1}$ can be cancelled from the sequence.

As ${\bf d}(\alpha)>\hat{0}$, there exists a missing face in $\Delta_i$ of dimension at least $2$, and let $e$ be an edge contained in it. Denote by $\Delta'$ the complex obtained from $\Delta_i$ by subdividing at $e$. In the sequence $\alpha$ replace
$\Delta_{i}$ by three consecutive complexes $(\Delta_i,\Delta',\Delta_i)$  to obtain a sequence $\alpha'$.
The sequence $\alpha'$ thus contains $(\Delta_{i-1},\Delta_i,\Delta',\Delta_i,\Delta_{i+1})$.
Since $e_1\neq e_2$, we have that $e\neq e_1$ or $e\neq e_2$.
We first consider the non-degenerated case with $e_1\neq e\neq e_2$.
\begin{figure}
    \begin{center}
      \begin{postscript}
      \psfrag{e1}{$e_1$}
      \psfrag{e}{$e$}
      \psfrag{a}{\mbox{}\hspace{2mm}$\longrightarrow$}
      \psfrag{T}{$T$ in $\Delta_{i-1}$}
      \includegraphics[width=.125\linewidth]{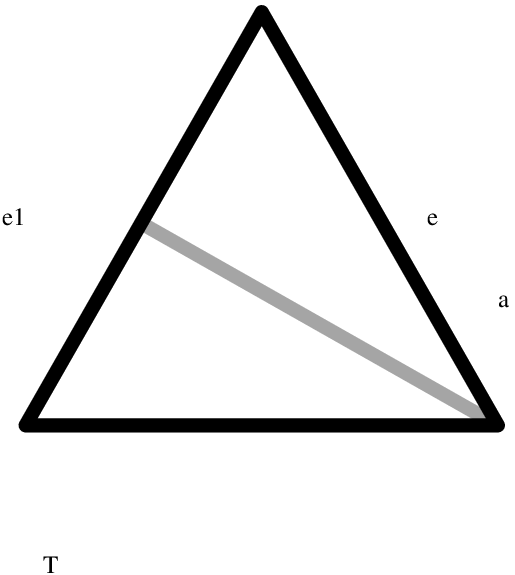}\hspace{9mm}
      \psfrag{e1}{$e_1$}
      \psfrag{e}{$e$}
      \psfrag{a}{\mbox{}\hspace{2mm}$\longrightarrow$}
      \psfrag{T}{$T$ in $\Delta'_1$}
      \includegraphics[width=.125\linewidth]{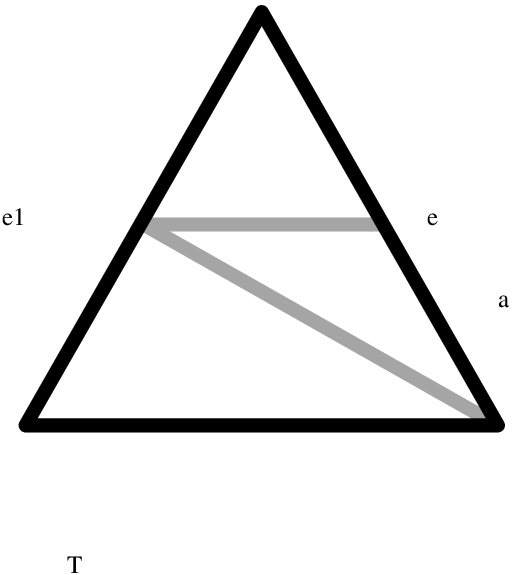}\hspace{9mm}
      \psfrag{e1}{$e_1$}
      \psfrag{e}{$e$}
      \psfrag{a}{\mbox{}\hspace{2mm}$\longrightarrow$}
      \psfrag{T}{$T$ in $\Delta'_2$}
      \includegraphics[width=.125\linewidth]{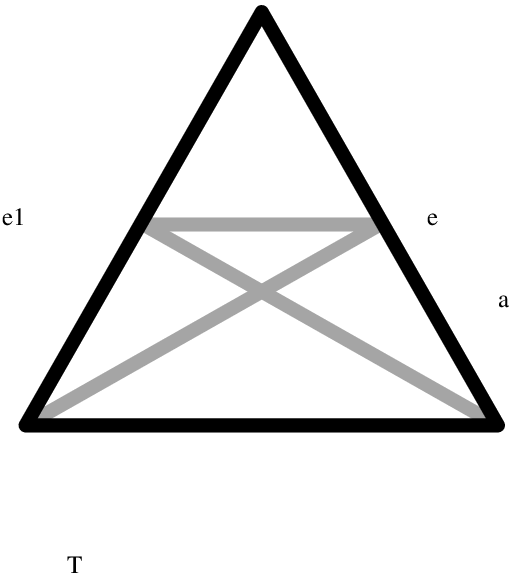}\hspace{9mm}
      \psfrag{e1}{$e_1$}
      \psfrag{e}{$e$}
      \psfrag{a}{\mbox{}\hspace{2mm}$\longrightarrow$}
      \psfrag{T}{$T$ in $\Delta'_3$}
      \includegraphics[width=.125\linewidth]{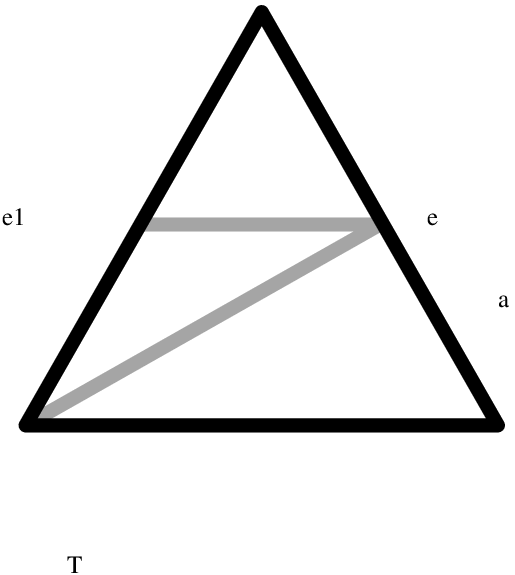}\hspace{9mm}
      \psfrag{e1}{$e_1$}
      \psfrag{e}{$e$}
      \psfrag{T}{$T$ in $\Delta'_4=\Delta'$}
      \includegraphics[width=.125\linewidth]{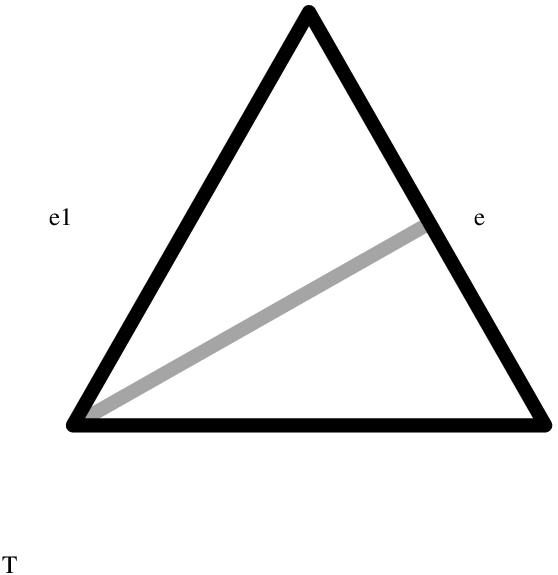}
      \end{postscript}
    \end{center}
    \caption{Sequence of five complexes $(\Delta_{i-1},\Delta'_1,\Delta'_2,\Delta'_3,\Delta'_4=\Delta')$.
      \label{fig:sequence-5-complexes}}
  \end{figure}

Case 1: $e$ and $e_1$ are not contained in a common $2$-face of $\Delta_i$.
Then the two subdivisions, at $e_1$ and at $e$, commute (e.g.\ \cite[Corollary 10:2a]{Alexander1930}). Replace in $\alpha'$ the part $(\Delta_{i-1},\Delta_i,\Delta')$ by the one obtained by commuting the subdivisions,
$(\Delta_{i-1},\Delta'',\Delta')$, and note that $d(\Delta'')\leq d(\Delta')<d(\Delta_i)$ by Lemmas \ref{lem:d(Delta)} and \ref{lem:edge-sd(flag)}.

Case 2: otherwise, $e$ and $e_1$ are in a (unique) $2$-face $T$, and replace in $\alpha'$ the part $(\Delta_{i-1},\Delta_i,\Delta')$ by
a sequence of five complexes $(\Delta_{i-1},\Delta'_1,\Delta'_2,\Delta'_3,\Delta'_4=\Delta')$ as induced by
the subdivisions of $T$ illustrated in Figure~\ref{fig:sequence-5-complexes},
see also \cite[Figure 1A]{DaSilva-Karu}. Note that $d(\Delta'_j)<d(\Delta_i)$ for $1\leq j\leq 4$ as each $\Delta'_j$ is obtained from $\Delta_i$ by a sequence of edge subdivisions that include $e$.

For the part $(\Delta',\Delta_i,\Delta_{i+1})$ of $\alpha'$
we do a similar replacement; resulting in a sequence $\alpha''$ from $\Delta$ to $\Gamma$ with ${\bf d}(\alpha'')<{\bf d}(\alpha)$.

In the degenerated cases $e_1=e\neq e_2$ and $e_1\neq e= e_2$ we first cancel the degenerated part $(\Delta_i,\Delta')$ and ($\Delta',\Delta_i$) from the sequence and then either execute Case 1 or Case 2 on the non-degenerated part, respectively.

Thus, after repeating the replacement process finitely many times we arrive at a sequence $\beta$ with ${\bf d}(\beta)=\hat{0}$, as desired.
\end{proof}

\section{Extremal flag spheres}\label{sec:apps}
Barnette's lower bound theorem for simplicial polytopes and spheres \cite{Barnette73:Graphtheoremsformanifolds,Barnette:LBT-73} follows from the inequality on face numbers of the $1$-skeleton for all simplicial spheres:
$$g_2:=f_1-df_0+\binom{d+1}{2} \geq 0,$$
where $d-1$ is the dimension of the sphere, and $f_i$ the number of $i$-dimensional faces. This reduction is known as McMullen--Perles--Walkup reduction (MPW).
Barnette proved that equality is attained if and only if the simplicial polytope is \emph{stacked}, and Kalai extended this result to all \emph{homology spheres} \cite{Kalai-LBT}.

Stronger lower bounds for the case where the homology spheres are \emph{flag} were conjectured in \cite[Conjecture 1.4]{Nevo-Missing}, and a reduction similar to MPW was shown \cite[Proposition 3.2]{Nevo-Missing} to the following inequality, for all flag homology $(d-1)$-spheres (same notation as above):
$$\gamma_2:=f_1-(2d-3)f_0+2d(d-2)\geq 0.$$
This inequality is part of Gal's conjecture that the entire 
\emph{$\gamma$-vector} $(\gamma_0,\gamma_1,\ldots,\gamma_{\lfloor\frac{d}{2}\rfloor})$ of flag homology $(d-1)$-spheres is nonnegative \cite{Gal},
where the $\gamma$-vector is defined by the polynomial equation 
$\sum_{i=0}^{\lfloor\frac{d}{2}\rfloor}\gamma_i t^i(t+1)^{d-2i}=(t-1)^d \sum_{i=0}^d f_{i-1}(\frac{1}{t-1})^i$.
Here, we will conjecture a characterization of the flag homology spheres with $\gamma_2=0$.

The conjecture below should be thought of as describing the flag analogues of stacked spheres. Further, to prove $\gamma_2\geq 0$ it is enough to consider flag spheres were each edge belongs to an induced $4$-cycle
(i.e., the subgraph induced by the vertices of the cycle is exactly the $4$-cycle). Call these triangulations \emph{minimal};  see the proof of Corollary~\ref{cor:cumputer_program} for details.
The conjecture below suggests that for $d>3$ minimal flag $(d-1)$-spheres, different from the octahedral sphere, must have $\gamma_2>0$.

\begin{conjecture}\label{conj:gamma_2=0}
Let $d\geq 4$ be an integer and $\Delta$ be a flag simplicial $(d-1)$-sphere. Then the following are equivalent:
\begin{compactitem}
\item[(i)] $\gamma_2(\Delta)=0$.
\item[(ii)] There is a sequence of edge contractions from $\Delta$ to the boundary of the $d$-dimensional cross polytope, i.e., to the octahedral $(d-1)$-sphere, such that all complexes in the sequence are flag spheres, and the link of each edge contracted is the octahedral $(d-3)$-sphere.
\end{compactitem}
\end{conjecture}
Part (ii) is the flag analog of stackedness: indeed, it is not difficult to see that a simplicial $(d-1)$-sphere $\Delta$ is stacked if and only if there is a sequence of edge contractions from $\Delta$ to the boundary of the $d$-simplex such that all complexes in the sequence are simplicial spheres, and the link of each edge contracted is the boundary of a $(d-2)$-simplex.

We remark that the implication $(ii)\Rightarrow (i)$ is easy: recall $\gamma_1:=f_0-2d$, then
for an edge contraction yielding $\Delta'=\Delta/e$ one has $\gamma_2(\Delta)=\gamma_2(\Delta')+\gamma_1(\lk_{\Delta}(e))$. As shown in \cite{Gal, Meshulam}, 
$\gamma_1\geq 0$ for all flag $(d-1)$-spheres, and  
the only flag spheres for which $\gamma_1$ vanishes are octahedral.

Thus, assuming $\gamma_2(\Delta)=0$ and existence of a sequence of edge contractions from $\Delta$ to the octahedral $(d-1)$-sphere with all complexes in the sequence flag spheres, implies that the links of the edges contracted must be octahedral spheres.

Conjecture \ref{conj:gamma_2=0} holds for the interesting subclass of (dual complexes of) \emph{flag nestohedra}, as Volodin \cite{Volodin}  showed they can be obtained from the octahedral sphere by a sequence of edge subdivisions.


To test the implication $(i)\Rightarrow (ii)$ in Conjecture \ref{conj:gamma_2=0}, as well as Gal's conjecture $\gamma_2\geq 0$, we run the following computer program.
\begin{enumerate}
\item{} Start with the octahedral $(d-1)$-sphere ($d\geq 4$),
\item{} perform at random (for some number of rounds) either an edge subdivision or a contraction of an edge which is in \emph{no} induced $4$-cycle (we call such contractions \emph{admissible}),
\item{} check if $\gamma_2\geq 0$ and
\item{} once $\gamma_2=0$ is reached, perform admissible edge contractions only as long as possible and check if the resulting flag sphere is the octahedral sphere.
\item{} repeat: go back to (2).
\end{enumerate}

\begin{corollary}\label{cor:cumputer_program}
Fix $d\geq 4$. Our computer program searches exactly through the entire space of $(d-1)$-dimensional flag PL spheres.
\end{corollary}
\begin{proof}
First note that the condition on admissible edge contractions guarantees that all the complexes obtained are flag.
This is well known to experts. As we could not find a reference, here is a proof.

Indeed, for an admissible contraction of edge $\{a,b\}$ in a flag complex $\Delta$, to a new vertex $v$, the resulting complex
\begin{eqnarray*}
\Delta' & \!:=\! & \{F\in\Delta: a,b\notin F\} \\
             &     & \cup\,\, \{F\cup\{v\}:\ F\cap \{a,b\}=\emptyset\  \rm{and\ either}\  F\cup\{a\}\in \Delta\ \rm{or} \ F\cup\{b\}\in \Delta\}
\end{eqnarray*}
has no missing faces of dimension larger than $1$.  First of all, $\Delta' $ has no missing triangles, since
otherwise if $F\uplus\{v\}$ is a missing triangle in $\Delta'$ with $|F|=2$, then $F\in\Delta$, but $a$ and $b$ can not be neighbors in $\Delta$ of both vertices of $F$, from which it follows that the edge $\{a,b\}$ is in an induced $4$-cycle, which was excluded.

Thus, suppose that $F\in \Delta$, $|F|>2$ and $F\uplus\{v\}$ is a missing face in $\Delta'$. We will show that one of $a,b$ is a neighbor of all vertices of $F$ in the $1$-skeleton of $\Delta$, which implies, as $\Delta$ is flag, that
$F\uplus\{v\}\in\Delta'$, a contradiction. If $b$ is not a neighbor of some $u'\in F$ then as $\partial(\overline{F\uplus\{v\}})\subseteq \Delta'$ we conclude that for any $u'\neq u \in F$, $(F\setminus\{u\})\cup\{a\}\in \Delta$.
As $|F|>2$ we get that $a$ is a neighbor of all elements of $F$ in $\Delta$ and thus $F\uplus\{a\}\in\Delta$, hence $F\uplus\{v\}\in\Delta'$, a contradiction.
We conclude that $\Delta'$ is flag.

In particular, the edges contracted satisfy the link-condition $$\lk(\{a,b\})=\lk(a)\cap\lk(b),$$ thus the contractions preserve the PL type of the sphere \cite{Nevo-VK}; clearly the (stellar) edge subdivisions preserve the PL type as well.
Note that the inverse of an edge subdivision on flag complexes is a special case of an admissible edge contraction. Thus,
Theorem~\ref{thm:main} finishes the proof.
\end{proof}

We now turn to a conjecture on the extremal examples for upper bounds.
Let $T(r,n)$ be the complete $r$-partite graph on $n$ vertices with the parts as equal size as possible. Tur\'{a}n showed that this graph has more edges than any other graph on $n$ vertices without an $(r+1)$-clique. The number of $i$ cliques in $T(r,n)$, denoted $f_{i-1}(r,n)$, can be easily computed and is roughly $~\binom{r}{i}(\frac{n}{r})^i$.

In \cite[Conjecture 6.3]{Nevo-Petersen} it was conjectured that for any flag homology sphere $\Delta$, $\gamma(\Delta)$ is the $f$-vector of some balanced complex. In particular, from the characterization of such $f$-vectors \cite{FFK} it would follow that if $\Delta$ is $(d-1)$-dimensional with $n$ vertices then
$$\gamma_i(\Delta)\leq f_{i-1}(\lfloor\frac{d}{2}\rfloor, n-2d) $$ for all $2\leq i\leq \lfloor\frac{d}{2}\rfloor$ (equality for $i=0,1$ is clear).
What can be said about the case of equality?

\begin{conjecture}\label{conj:flagUBC}
Let $d\geq 4$ be even and $\Delta$ be a flag simplicial $(d-1)$-sphere on $n$ vertices.
Then the following are equivalent:
\begin{compactitem}
\item[(i)] $\gamma_i(\Delta)=
f_{i-1}(\frac{d}{2}, n-2d)$ for some $2\leq i\leq \frac{d}{2}$.
\item[(ii)] $\Delta$ is the join of\, $\frac{d}{2}$ cycles of as equal length as possible.
\end{compactitem}
\end{conjecture}

Clearly (ii) implies (i); further, among joins of $\frac{d}{2}$ cycles with a total of $n$ vertices, the join where the cycles are as equal length as possible is the unique maximizer of each of $\gamma_i$ for $2 \leq i\leq \frac{d}{2}$.

We remark that this conjecture is in contrast to the usual upper bound theorem for simplicial polytopes (McMullen \cite{McMullen:MaximumNumberFacesConvexPolytope-70}) and spheres (Stanley \cite{Stanley:CohenMacaulayUBC-75}), where equality is attained by numerous examples, namely by all neighborly polytopes and spheres.
For $d=4$, Conjecture~\ref{conj:flagUBC} follows
from a conjecture of Gal \cite[Conjecture 3.2.2]{Gal}. Very recently, the case $d=4$ of the conjecture was confirmed when $\gamma_1$ is large enough~\cite{Adamaszek-Hladky}, compare also \cite[Conjecture 5.1]{Adamaszek-Hladky}.

Our computer experiments support Conjecture \ref{conj:gamma_2=0} as well as Conjecture~\ref{conj:flagUBC}.
For our search, we used a variation of the bistellar flip program \texttt{BISTELLAR} \cite{BjoernerLutz2000,Lutz_BISTELLAR},
where we replaced the standard bistellar flips by edge subdivisions and admissible edge contractions.

\medskip

\textbf{Acknowledgments}:
We thank Micha\l\ Adamaszek, Karim Adiprasito, Gil Kalai, Ilya Tyomkin, and Vadim Volodin
for helpful discussions and remarks.
We are also indebted to the anonymous referee for valuable comments with respect to the
presentation of the paper.

\bibliography{.}

\end{document}